\documentclass[12pt, a4paper]{article}
\usepackage{amsfonts, amssymb}
\usepackage{graphicx}
\usepackage{latexsym,amsmath,color,array}

\usepackage{tikz}
\usetikzlibrary{automata, positioning,shapes,snakes}

\usepackage{natbib}
\usepackage{enumerate}
\usepackage{bbding}
\usepackage{amssymb}
\usepackage{amsmath}
\usepackage{graphicx}
\usepackage{amsmath}
\usepackage{bbm}
\usepackage{mathtools}
 \usepackage{color}
 \usepackage{array}
 \usepackage{subfigure}
 \usepackage{multirow}
 \usepackage{makecell}
 \usepackage{colortbl}

\def\1{\mathbb{I}}

\usepackage{amsthm}
\newtheorem{theorem}{Theorem}
\newtheorem{corollary}{Corollary}

\newcounter{example}
\newenvironment{example}[1][]{\refstepcounter{example}\par\medskip \noindent \textbf{Example~\theexample. #1} \rmfamily}{\medskip}

\newcounter{remark}
\newenvironment{remark}[1][]{\refstepcounter{remark}\par\medskip \noindent \textbf{Remark~\theremark. #1} \rmfamily}{\medskip}

\newcounter{algo}
\newenvironment{algo}[1][]{\refstepcounter{algo}\par\medskip \noindent \textbf{Algorithm~\thealgo. #1} \rmfamily}{\medskip}

\newcounter{step}
\newenvironment{step}[1][]{\refstepcounter{step}\par\medskip \noindent \textbf{\textit{Step~\thestep. #1}} \rmfamily}{\medskip}
\begin{document}

\title{ On the Non-uniqueness of Representations of Coxian Phase-Type Distributions}
\author{Jean Rizk\footnote{University of Limerick; Jean.Rizk@ul.ie} \hspace{3cm}
Kevin Burke\footnote{University of Limerick; Kevin.Burke@ul.ie} \hspace{3cm}
Cathal Walsh\footnote{University of Limerick; Cathal.Walsh@ul.ie}  }
\date{\today}

\maketitle

\begin{abstract}
 Parameter estimation in Coxian phase-type models can be challenging due to their non-unique representation leading to a multi-modal likelihood. Since each representation corresponds to a different underlying data-generating mechanism, it is of interest to identify those supported by given data (i.e., find all likelihood modes). The standard approach is to simply refit using various initial values, but this has no guarantee of working. Thus, we develop new properties specific to this class of models, and employ these to determine all the equivalent model representations. The proposed approach only requires fitting the model once, and is guaranteed to find all representations.

\smallskip


\end{abstract}

\qquad

\newpage

\section{Introduction\label{intro}}
A Coxian phase-type (CPH) distribution of order $(n)$, which we denote by  $n$-CPH, describes duration until absorption in terms of a continuous time Markov process consisting of a sequence of ($n$) transient  latent phases and one absorbing state.  The process starts in the first phase and progresses sequentially through the other phases with a probability of exiting (to the absorbing state) from any phase. The Coxian Markov model is illustrated in Figure \ref{fig1}. 
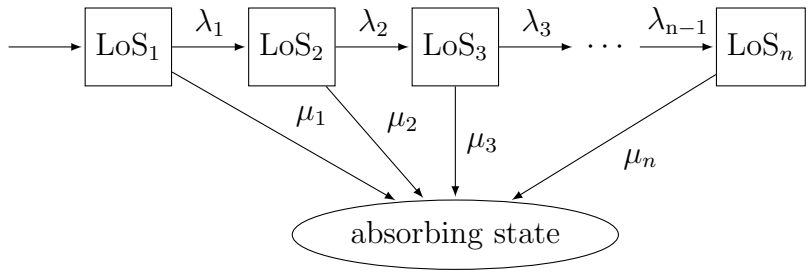
\begin{figure}[h]
\centering
\begin{tikzpicture}
        \tikzset{node style/.style={state, 
                                    fill=white!40!white,
                                    rectangle}}
    
        \node[node style]               (I)   {LoS$_1$};
        \node[node style, right=of I]   (II)  {LoS$_2$};
        \node[node style, right=of II]  (III) {LoS$_3$};
        \node[draw=none,  right=of III]   (k-n) {$\cdots$};
        \node[node style, right=of k-n]  (n)   {LoS$_n$};
        
        \node[draw=none, left=of I] (above) {};
        \node[node style,ellipse,fill=white!60!white,minimum height=0.7cm,minimum width=1cm,node distance=1.5cm, below=of III] (q) {absorbing state};
    \draw[>=latex,auto=left,every loop]
         (above)   edge node {}        (I) 
       (I)  edge node {$\mu_{1}$} (q)
      (II)  edge node {$\mu_{2}$}  (q)
       (III) edge node {$\mu_{3}$} (q)
       (n)   edge node {$\mu_n$}   (q)
  (I) edge node {$\lambda_{\mathrm{1}}$}(II)
  (II)  edge node {$\lambda_{\mathrm{2}}$}(III)
  (III) edge node {$\lambda_{\mathrm{3}}$} (k-n)
   (k-n)  edge node {$\lambda_{\mathrm{n-1}}$}  (n);
\end{tikzpicture}
 \caption{An illustration of the Coxian Markov model.}\label{fig1}
\end{figure}

The $\lambda$ parameters   describe the transition rates through the transient states. The $\mu$ parameters describe the transition rates from the transient states to the absorbing state. The total time spent in the system is broken down into the distinct phases, representing different stages of the entire process. For instance, in healthcare applications, the length of stay (LoS) in a phase  may represent the time spent by a patient in a particular stage of care or a disease state.

 Coxian distributions are a subclass of phase-type (PH) distributions.  The non-uniqueness of PH distributions \citep{OCinneide1989,telek2007}, as well as their overparametrisation, has encouraged several researchers to investigate the minimal PH representation problem, which is the determination of the minimal number of phases for a given PH distribution. For example, \cite{cumani1982} showed that every  acyclic PH distribution (APH) (a PH representation with a triangular generator matrix) can be transformed into an equivalent minimum-parameter form known as the canonical form, with a bidiagonal generator matrix. \cite{mocanu1999} proposed a transformation of a general PH representation to a monocyclic representation where the generator matrix remains bidiagonal on the matrix block level. A smaller representation leads to a shorter computational time in random-variate generation \citep{reinecke2010}, and the reduced number of parameters improves estimation performance.

 A Coxian distribution has a bidiagonal structure, the generator matrix is such that $a_{ii}<0$, $a_{i,i+1}>0$ and $a_{i,i+1}$ is not necessarily equal to $-a_{i,i}$. The representation where $a_{ii}= -a_{i,i+1}$ is called hypoexponential or generalised Erlang representation, where the process starts in the first phase and progresses sequentially through the other phases to exit from the last phase. The generalised Erlang Markov model is illustrated in Figure \ref{fig2}. 
 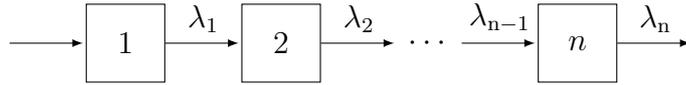
\begin{figure}[t]
\centering
\begin{tikzpicture}
        \tikzset{node style/.style={state, 
                                    fill=white!40!white,
                                    rectangle}}
    
        \node[node style]               (I)   {$1$};
        \node[node style, right=of I]   (II)  {$2$};
        \node[draw=none,  right=of II]   (k-n) {$\cdots$};
        \node[node style, right=of k-n]  (n)   {$n$};
        \node[draw=none, right=of n] (n+1)   {};
        
        \node[draw=none, left=of I] (above) {};
       
    \draw[>=latex,auto=left,every loop]
         (above)   edge node {}        (I) 
      
  (I) edge node {$\lambda_{\mathrm{1}}$}(II)
  (II) edge node {$\lambda_{\mathrm{2}}$} (k-n)
   (k-n)  edge node {$\lambda_{\mathrm{n-1}}$}  (n)
    (n)  edge node {$\lambda_{\mathrm{n}}$}  (n+1);
\end{tikzpicture}
\caption{An illustration of the generalised Erlang Markov model.}\label{fig2}
\end{figure}
The Coxian representation is already in its minimal form; however, although the dimension of the parameter vector cannot be reduced, non-uniqueness still arises. There may exist representations which have different non-zero diagonal and superdiagonal entries in the generator matrix, but which yield the same marginal distribution. Each set of parameters result in a different LoS in each phase, which means a different physical interpretation of the underlying mode.  Despite convergence of the optimisation algorithm, when multiple fits to data are performed from different starting values, they produce different sets of parameter estimates that correspond to the same maximum likelihood (Marshall and Zenga, 2012). This is a well known identifiability problem when fitting an $n$-CPH distributions and leads to optimisation algorithms running from one Coxian representation  into another Coxian representation of the same order $n$. For the rest of the paper, we refer to these representations as ''equivalent CPH distributions".

 While the the non-uniqueness of Coxian distributions poses some optimisation issues in itself (i.e., the multi-modal likelihood surface), it is necessary to select a given representation, and this choice cannot be made on the basis of the likelihood function alone. For example, in healthcare modelling, the choice could be based on medical experts' knowledge with a post hoc examination of the phases identified in the model-fitting process. However, another problem often occurs, which has not received attention in the literature:  the fitting algorithms might not find all the possible representations. In this case we become constrained to select from the representations that appear during the fitting procedure;  moreover,  the expectations of experts might not match  the outputs of the fitting process. Thus, it is necessary to identify all possible latent structures which yield the same marginal model. 

In the existing literature, a given representation is chosen (perhaps arbitrarily) without discussion of the logic underpinning  this final choice; a typical, and still arbitrary, choice is to simply select the representation which places the LoS's in ascending/descending order (which is equivalent to imposing constraints on model parameters). These strategies ignore the fact that other representations can also be feasible. It may be that some researchers simply reorder the phases according to LoS manually even if such a representation did not appear in the model fitting - such practice would yield erroneous conclusions as not all LoS permutations are necessarily feasible as we show in the sequel. Of course, one might diligently refit the model many times from a large variety of initial values hoping to uncover all equivalent  representations, but this does not guarantee that all will be found, and, moreover, this approach is computationally intensive. 

 The non-uniqueness of representations of two CPH distributions is in fact a special case of the non-uniqueness of two PH distributions of the same order that was addressed and proven by  \cite{OCinneide1989} and \cite{telek2007}.  \cite{OCinneide1989} addressed the non-uniqueness question by introducing two properties. (a) the phase-type ``simplicity" which concerns the possibility that a given phase-type distribution may have two different representations in terms of the same Markov chain, and (b) the phase-type ``majorization" which concerns the possibility that one Markov chain may provide representations for all the distributions represented by another Markov chain.  \cite{telek2007} addressed the same questions of non-uniqueness but presented alternative (somewhat more direct) theory than that of \cite{OCinneide1989}. Nonetheless, their work only enables one to verify if two pre-defined CPH distributions are equivalent or not. 
 
 \cite{reinecke2010} considered the identification of alternative representations in the setting of the generalised Erlang distribution (defined above) which, like the Coxian distribution, is another minimal form PH distribution. There, one must simply permute the diagonal elements of the generator matrix to produce a new representation; the superdiagonal elements are the negation of diagonal elements, and, therefore, these permute together. However, since the Coxian form is more general than that of the generalised Erlang, generating alternative representations is not so straightforward; this, perhaps, is the reason it has not been considered previously in the literature. We note, however, that \cite{horvath2012, horvath2016} consider alternative representations for the feedback Erlang distribution (a different generalisation of the generalised Erlang), which, like the generalised Erlang case, can be produced straightforwardly due to the fact that the generator matrix is in the Erlang form but on the matrix block level.

In general, finding an equivalent representation (with generator $Q_b$) to a PH distribution (with generator $Q_a$), can be partially achieved by solving the transformation equation $Q_aM=MQ_b$ \citep{cumani1982} for matrices $M$ and $Q_b$.  In this paper, we use this equation to  explore the structure of the transformation matrix $M$ when $Q_a$ and $Q_b$ become Coxian generators, and we prove that the equation is not needed to find an equivalent Coxian representation. An equivalent representation can simply be found by matching the moments of the two CPH distributions.

The purpose of our current work is identifying the feasibility of all the equivalent  representations of a given CPH distribution without the need to perform multiple refits.  As mentioned above, finding and checking the feasibility of these representations is important from the perspective of the physical interpretation of the model. To our knowledge, no effective mechanism for identifying all representations of a given CPH distribution has been developed. For this reason, we add to the existing literature and direct the focus towards CPH distributions, developing new properties which are specific to these distributions. These properties will be employed  as a mechanism for identifying all the possible representations of a given CPH distribution.

The paper is organised as follows. In Section 2 we give an overview of PH and CPH distributions and their applications before formally defining CPH distributions, their characteristics, and details on fitting procedures. Section 3 contains a detailed discussion on the identifiability problem, and our development of some new properties of CPH distributions which we make use of in Section 4 where we propose a method for identifying all representations of a given CPH distribution.  Finally, our conclusions are presented in Section 5.

\section{ Coxian phase-type distributions}
\subsection{Background}
PH distributions were first introduced by \cite{neuts1975}. They describe the time to absorption of a finite Markov chain in continuous time, where there is a single absorbing state and the stochastic process starts in a transient state. PH distributions became more attractive in the area of queue modelling and survival analysis. These areas generally exhibit data with skewed and heavy tailed distributions which make PH distributions particularly suitable. Furthermore,  these distributions have the ability to offer superior fit compared to the alternative distributions such as lognormal, Weibull, gamma, Pareto, or Burr distributions \citep{Faddy2009, Marshall2014}. In fact, PH distributions are dense in the class of all non-negative distributions, i.e., any distribution with a non-negative density function can be approximated to arbitrary precision by a PH distribution \citep{buchholz2014}.

 CPH distributions \citep{cox1955} are a subclass of PH distributions. In recent decades, most researchers have avoided using general PH distributions because they are overparametrised. They are highly redundant as the number of model parameters is greater than the degrees of freedom of the distribution function. The representation of an $n$-PH distribution ($n$ is the number of phases) has in general $n^2+n$ parameters, and its corresponding distribution function has $2n-1$ degrees of freedom \citep{cumani1982}. Using an $n$-CPH distribution reduces the number of parameters to $2n-1$, which makes it non-redundant, while typically still providing an excellent fit to the data. However, even with the reduced number of parameters required for the CPH  distribution, estimation can still be problematic due to the non-linear expression and non-unique representations of the distribution.
 
Based on the generalisation of Erlang's method of stages \citep{erlang1917}, CPH distributions are in fact a mixture of Hypoexponential distributions or generalised Erlang distributions \citep{augustin1982}. Note that a Hypoexponential distribution is a convolution of independent but non-identical exponential distributions. The exponential structure of the CPH distributions make them tractable and well-suited for numerical computations \citep{ishay2002thesis}.
 
  The use of CPH distributions has become increasingly more popular in the area of survival analysis particularly in healthcare applications. To mention a few, \cite{faddy1994} fitted CPH distributions of increasing order to the length of treatment for patients at risk of suicide.  \cite{vasilakis-adelle2005} modelled the LoS in hospital of stroke patients over the age of 65 in the UK. \cite{Marshall2014}  used the CPH distribution with covariates to model the LoS of geriatric patients in Emilia Romagna hospitals. \cite{zhu2018} also used CPH models with covariates to analyse the LoS of respiratory patients in emergency department. CPH distributions have also proven useful in the area of queueing theory,  where they have the ability to provide a generalisation of arrival and service processes beyond the widely used Poisson processes, by keeping a tractable and conceptually easy format for the state probabilities \citep{johnson1993_a, agarwal2007}.

\subsection{The model}
Consider a finite Markov process $\{ X(t); t \geq  0\}$  defined  in continuous time with discrete states $\{1, \dots, n, n+1 \}$. Here, states 1 to $n$  are latent transient states, while state $n+1$ represents the absorbing state. Let $p_i=\text{pr}(X(0) = i)$ be the probability of starting in the transient state $i$, for $i=1,\dots, n$. 
In a Coxian model the process starts in the first phase and hence the initial distribution is $p=(p_1,\dots,  p_n)=(1,0, \dots, 0) $, which is a row vector. Let the column vector $q=(\mu_1, \dots, \mu_n)^T\in \mathbbm{R}_{\geq 0}^n$ be the absorbing rate  vector, where $\mu_i$ is the rate of absorption to state $n+1$ from state $i$. The intensity matrix, $R$, of the process is 
\begin{align*}
{ R}=\begin{bmatrix}
        Q & \quad q \\
      0_{1\times n} &  \quad 0\\
      \end{bmatrix},
\end{align*}
where $0_{1\times n}$  is an $n$-dimensional row of zeros, and  $Q$,  an $n \times n$ matrix, is the phase-type generator of the n-CPH distribution, and is an upper bidiagonal given by 
\begin{align}
Q=\begin{bmatrix}
        -(\lambda_1 + \mu_1) & \lambda_1 & \quad 0  & \quad \quad&\cdots & \quad 0& \quad 0\\
      0 &  -(\lambda_2 + \mu_2) & \quad\lambda_2 &  &\cdots &\quad 0& \quad 0\\
      \vdots & \vdots & \quad \vdots& & &\quad \vdots&\quad\vdots\\
      0 &  0 & \quad 0 & & \cdots  & \quad -(\lambda_{n-1} + \mu_{n-1})  & \quad \lambda_{n-1}\\
       0 & 0 & \quad 0 &&\cdots  &\quad 0& \quad -\mu_n 
       \end{bmatrix}. \label{matrixQ}
\end{align}

 Since every row in $R$ sums to zero, it follows that
\begin{equation}
q = -Q\mathbbm{1},\label{vecq}
\end{equation}
where $\mathbbm{1}$ is an $n$-dimensional column vector of ones.
The pair $(p,Q)$ is sufficient to represent the CPH distribution. 

We denote by $T$ the random variable representing  the time until absorption. The  density function and the Laplace transform  of an $n$-CPH distribution are respectively given by $f(t)=p\exp(Qt)q$ and $ f^*(s)=p(sI - Q)^{-1}(-Q) \mathbbm{1}$. The latter is a rational function \citep{cox1955} with a denominator of degree $n$ and a numerator of degree $n-1$. The number of non-trivial coefficient of the Laplace transform,  $2n-1$, is the degrees of freedom of the distribution function  \citep{cumani1982}. \\  The $rth$ moment is 
\begin{align}
\label{moment}
 E\big[T^r\big]=r!p(-Q^{-1})^r\mathbbm{1}.
\end{align}
The matrix exponential, $\exp(Qt)=\sum_{r=0}^{\infty}(Qt)^r/r!$, is evaluated numerically. More details on computing matrix exponentials can be found in \cite{moler1978}.

The time spent in phase $k$ ($k=1, \dots, n$), denoted by $T_k$, is a random variable that is the minimum of two independent exponential random variables with parameters $\lambda_k$ and $\mu_k$. The random variable $T_k$ is in turn exponential with rate $\lambda_k+\mu_k$. The expected length of stay in  phase $k$  is
LoS$_k=E[T_k]=1/(\lambda_k +\mu_k)$, where, by definition, $\lambda_n \equiv 0$. Note that, the LoS$_k$ values are in fact the negative reciprocal of diag($Q)=(Q_{11}, \dots, Q_{nn})$, the diagonal elements of the generator $Q$.  

The probability  of exiting phase $k$, $\pi_k$, can be calculated by again using the density function of the exponentially distributed time spent in phase $k$ \citep{marshall2004}: 
\begin{align*}
\pi_k=\frac{\mu_k}{\lambda_k + \mu_k}.\prod_{j=1}^{k-1} \Big(\frac{\lambda_j}{\lambda_j+\mu_j} \Big),\quad  k=1,\dots,n
\end{align*}
where $\lambda_n \equiv 0$. \\

\subsection{Parameter estimation}

Least squares \citep{faddy1990,faddy1993}, moment matching \citep{johnson1993_b, schmickler1992} and maximum likelihood \citep{faddy1994,asmussen1996,faddy1999,perez2003,marshall2009,marshall2012exper} are the three main techniques that have been used to estimate the parameters when fitting a CPH distribution. The efficiency of the three  methods has been discussed and more details can be found in \cite{lang1996}, \cite{marshall2009,marshall2012exper}.

 The most common approach taken is the  maximum likelihood. A variety of optimisation techniques have been developed over the years aiming to minimise the CPH log-likelihood function,
\begin{align*}
\ell(\theta| t)=\sum_{i}\log \big[ f(t_i|\theta)\big]=\sum_{i}\log \big( pe^{Qt_i}{q\big)},
  \end{align*}

where $\theta=(\lambda_1, \dots, \lambda_{n-1}, \mu_1, \dots, \mu_n)$ is the vector of parameters to be estimated.\\

Due to the non-linearity of the multi-dimensional likelihood function, numerical optimsations have to be employed. \cite{asmussen1996} developed the EMPht program in the \texttt{C} programming language which is based on the expectation-maximization  algorithm \citep{dempster1977}. The EMPht algorithm is an iterative method that is based on approximating a non-negative continuous distributions by a phase-type distribution by minimising the information divergence (the Kullback-Leiber information) which can be considered as an infinitesimal analogue of maximising the log-likelihood function \citep{olsson1998empht}.
 
The most commonly used algorithm in this CPH context is the Nelder-Mead simplex method \citep{nelder1965} that is available in MATLAB Optimisation Toolbox \citep{MATLAB2015} under the function \texttt{fminsearch} or in \texttt{R} programming language \citep{R} under the function \texttt{optim}; for details see \cite{faddy1994}, \cite{faddy1999}, \cite{perez2003} and \cite{marshall2012exper}.\\
\cite{marshall2009} alternatively employed the Sequential Quadratic Programming (SQP) method that proved to exhibit a high rate of convergence. In this method, a Quadratic Programming subproblem is solved at each iteration along with an update on the estimate of the Hessian of the Lagrangian. It is also found in MATLAB under the package \texttt{fmincon}. An advantage of using this  package is that it has the ability to fit within defined parameter constraints. This is useful in the case we have prior information about the model parameters. The prior could be deterministic \citep{xie2012} or distributional \citep{ausin2008bayesian} and setting constraints would help reducing the computational time.

As these optimisation methods are numerical in nature, a common problem we often encounter is that they strongly depend on the initial parameter values \citep{marshall2009}. In addition, they do not always converge. Thus, the typical strategy is to initiate the optimisation algorithm from a variety of initial values. To obtain the number of latent phases one should fit sequentially an increasing number of phases \citep{faddy1998}, starting with one phase, until little improvement in the fit to the data can be obtained by adding a new phase. The number of phases is typically determined by minimising the Akaike information criterion (\textsc{AIC}). 

\section{The identifiability problem}
The matrix representation of the CPH distribution makes fitting the  distribution a difficult optimisation problem \citep{buchholz2014}. Despite convergence of the optimisation process, there are multiple sets of parameters that yield the same maximum likelihood, i.e., the likelihood has multiple global maxima. This is due to the non-unique representation which can cause the optimisation algorithms to jump from one representation into another equivalent representation. We illustrate the non-uniqueness problem in the CPH distribution in the following example.

\begin{example}
\label{exp1}
We simulated a dataset with 5,000 observations from a 3-CPH distribution with the following parameter values: $\lambda_1$=0.55, $\mu_1$=0.003, $\lambda_2$=0.05, $\mu_2$=0.15 and $\mu_3$=0.1. These parameters are taken from \cite{payne2011}. Ten different sets of initial values were used in the fitting procedure. The process converged to two different sets of parameter estimates with the same log-likelihood value. The estimated parameters along with their corresponding log-likelihood and  LoS$_k$ values are shown in Table~\ref{table1}.

\begin{table}[t]
\caption{Two different sets of parameters corresponding to one 3-CPH distribution.}
\centering
 \begin{tabular}{cccccccccc}
 \hline
 &&&&&&&&&\\[-0.4cm]
$\ell{(\theta)}$ & $\hat{\lambda_1}$ &$\hat{\mu_1}$ &$\hat{\lambda_2}$ &$\hat{\mu_2}$ & $\hat{\mu_3}$&$LoS_1$ & $LoS_2$ & $LoS_3$ \\ 
&&&&&&&&&\\[-0.4cm]
\hline
-15924.19 & 0.570 & 0.001 &  0.029 &0.143&0.091&1.75&	5.82&10.98\\
 & (0.030)  & (0.002) &  (0.017) &(0.005) &(0.011) & &	&\\
-15924.19& 0.170 & 0.001 &  0.096  &0.474 & 0.091& 5.82& 1.75&10.98 \\
& (0.028)  & (0.002) &  (0.068) &(0.057) &(0.015) & &	&\\
&&&&&&&&\\[-0.4cm]
\hline
\end{tabular}
\label{table1}
\end{table}

The two estimated distributions are equivalent in the sense that  their density functions coincide exactly when plotted (not shown). Despite the equivalence of the two distributions, their corresponding parameters are different which result in different LoS$_k$ values. In fact the LoS$_k$ values are the same but permuted. Each permutation, however, leads to a different model interpretation. 

Note that in all the simulations that we have carried out,  the first absorbing rate $\mu_1$ remained invariant, i.e., this parameter does not change with the alternative representation (see Table~\ref{table1} for example). In fact, it turns out that this is a common  feature between two equivalent Coxian distributions which we prove it later in this section. 

\end{example}

In the remainder of this section, we discuss and prove the equivalence between two CPH distributions. We also develop new properties associated to these distributions. The properties include the invariance of the first absorbing rate $\mu_1$ as well as the permutation of LoS$_k$ values for two equivalent CPH distributions.

\begin{theorem}
\label{theorem1}
 Let $( p_a,Q_a)$ and $(p_b,Q_b)$ be two non-redundant PH distributions of the same order with density functions $f_a(t)=p_ae^{Q_at}q_a$ and $ f_b(t)=p_be^{Q_bt}q_b$, respectively.\\ Then $f_a(t) \equiv f_b(t)\iff$ $\exists$ a non-singular matrix $M$ such that:
(a) $\> p_aM=p_b$, 
(b) $\>Q_b=M^{-1}Q_a  M$  and 
 (c) $\>Mq_b=q_a$ .
\end{theorem}
\begin{remark} As mentioned previously, a distribution is said to be redundant when the number of model parameters is greater than the degrees of freedom of the distribution function. The latter is the number of non-trivial coefficient of the Laplace transform \citep{cumani1982}.
\end{remark}
\begin{remark}
This theorem is different to Theorem 1 of \cite{telek2007}. The difference is that we use the equivalence of the probability density functions  rather than the cumulative distribution functions. Doing so results in the appearance of the new property (c) which relates the absorbing rate vectors to each other, which we make use of when investigating the non-uniqueness of CPH representations.

\end{remark}

 \begin{proof} 
 ($\impliedby$) If $p_aM=p_b$, $\>Q_b=M^{-1}Q_a M$ and $\>M^{-1}q_a=q_b$ then, \\
$f_b(t)=p_be^{Q_bt}q_b=p_aMe^{M^{-1}Q_aMt}M^{-1}q_a=p_aMM^{-1}e^{Q_at}MM^{-1}q_a=p_ae^{Q_at}q_a=f_a(t)$.\\
($\implies$) For this implication, we follow along the lines of \cite{telek2007} to prove conditions (a) and (b).\\
If $f_a(t) \equiv f_b(t)$ then they have identical moments (Eq.(\ref{moment})),
\begin{equation}
p_a(-Q_a)^{-k}\mathbbm{1}=p_b(-Q_b)^{-k}\mathbbm{1}. \label{eqmoment1}
 \end{equation}
Note that If $( p,Q)$ is a non-redundant PH distribution then there is $(-Q)^{-1}=D^{-1}JD$ a Jordan decomposition of $(-Q)^{-1}$, normalised such that $D\mathbbm{1}=\mathbbm{1}$ \citep{telek2007}. From this result, Equation(\ref{eqmoment1}) becomes 
 \begin{align}
p_a(D_a^{-1}J_a^{-k}D_a)\mathbbm{1}&=p_b(D_b^{-1}J_b^{-k}D_b)\mathbbm{1}\nonumber\\
\Rightarrow \quad p_aD_a^{-1}J_a^{-k}\mathbbm{1}&=p_bD_b^{-1}J_b^{-k}\mathbbm{1}.\label{eqmoment2}
 \end{align}
The last equality follows from the fact that $D_a\mathbbm{1}= D_b\mathbbm{1}=\mathbbm{1}$. The non-redundancy  ensures that the Jordan decomposition of the generator matrix is such that all identical eigenvalues belong to the same Jordan block.   Since $(p_a,Q_a)$ and $( p_b,Q_b)$ are non-redundant and $J_a$ and $J_b$ are  Jordan matrices with ordered eigenvalues, then Equation(\ref{eqmoment2}) implies that 
\begin{equation}
J_a=J_b
\label{jordan}
\end{equation}
and $p_aD_a^{-1}=p_bD_b^{-1} \Rightarrow p_b=p_aD_a^{-1}D_b.$\\
 Let $M=D_a^{-1}D_b$, then,\\
$ (a) \> p_b=p_aD_a^{-1}D_b=p_aM$.\\
$ (b) \> M^{-1}Q_a M= D_b^{-1}D_aQ_aD_a^{-1}D_b=D_b^{-1}D_aD_a^{-1}J_a^{-1}D_aD_a^{-1}D_b=D_b^{-1}J_b^{-1}D_b=Q_b$.\\
$(c) \> Mq_b= -MQ_b\mathbbm{1} =- MM^{-1}Q_aM\mathbbm{1}=-Q_aM\mathbbm{1}=-Q_a\mathbbm{1}=q_a.$ \\
In the last step we used that 
\begin{equation}
\label{m1=1}
M\mathbbm{1}=\mathbbm{1} 
\end{equation}
 since   $M=D_a^{-1}D_b$,  $D_b\mathbbm{1}=\mathbbm{1}$ and  $D_a^{-1}\mathbbm{1}=\mathbbm{1}$. 
\end{proof} 

\begin{remark} A well known result from linear algebra is that two square matrices $C$ and $D$ are ``similar" if there exists a non singular matrix $M$ such that $D=M^{-1}C M$ or equivalently $MD=CM$.  This equivalence relation is satisfied in condition (b) of Theorem \ref{theorem1} above. Thus, we can say that if two non-redundant PH distributions are equivalent then their corresponding generator matrices are similar. The matrix $M$ is called the transformation matrix and must satisfy Equation (9), which means the rows of $M$ must sum to one.
\end{remark}
 
\begin{remark} The conditions presented in Theorem \ref{theorem1} only enable one to verify whether or not  two given CPH distributions are equivalent, by finding a matrix $M$ that transforms one representation into another. However, the aim of this work is to find all the possible representations of one given Coxian distribution, not to simply verify that two are equivalent. Thus, we now develop new properties specific to CPH distributions (in Corollaries 1 and 2, and Theorem 2) which form the basis of our mechanism for identifying all representations (given in Section 4).
\end{remark}
 
 \begin{corollary} 
 \label{corollary1}
If $( p_a,Q_a)$ and $( p_b,Q_b)$ are two equivalent $n$-CPH distributions, then the transformation matrix $M$ is lower triangular with first row equals $(1,0, \dots, 0)$. 
\end{corollary}

\begin{proof} 
Let $M$ be a matrix with elements denoted by $M=(m_{ij})_{i,j}$, and the generator matrices are given by
\begin{align}
\label{QaQb}
Q_a = \begin{bmatrix}
a_{11} & a_{12}& & 0\\
 & \ddots & \ddots & \\
&  & \ddots & a_{(n-1)n} \\
0 & &  & a_{nn} &
\end{bmatrix} \> \text{and} \> \> Q_b = \begin{bmatrix}
b_{11} & b_{12}& & 0\\
 & \ddots & \ddots & \\
&  & \ddots & b_{(n-1)n} \\
0 & &  & b_{nn} &
\end{bmatrix}.
\end{align} 
The diagonal and superdiagonal  elements of $Q_a$ and $Q_b$ are all non-zero. \\ Since the two equivalent distributions are Coxian then $p_a=p_b=(1,0,\dots, 0)$.  Using condition $(a)$ in Theorem \ref{theorem1}, we have 
\begin{align}
p_a M&=p_b \nonumber\\
(1,0, \dots,0)M&=(1,0,\dots, 0) \nonumber\\
\Rightarrow (m_{1j})_{j=1}^n&=(1,0,\dots, 0); \label{firstrow}
\end{align}
which means the first row of $M$ is $(1, 0, \dots, 0)$.\\
We now use condition $(b)$ in Theorem \ref{theorem1}, $MQ_b=Q_aM$. We have, \\
\begin{align*}
MQ_b=(B_{ij})_{i,j}&= 
\begin{cases}
m_{i(j-1)}b_{(j-1)j} + m_{ij}b_{jj} & \quad i=1,\dots,n;  j=2,\dots,n \\
      m_{ij}b_{jj} &  \quad i=1,\dots,n;  j=1\\
\end{cases} \\
  Q_aM=(A_{ij})_{i,j}&=
\begin{cases}
      a_{ii}m_{ij}+a_{i(i+1)}m_{(i+1)j} & \quad i=1,\dots,n-1; \> \>j=1,\dots,n\\
      a_{ii}m_{ij} & \quad i=n ;  j=1,\dots,n \\
\end{cases} 
\end{align*}
By equating the above we define a difference equation as follows: 
 
\begin{align}\label{diffequ1}
a_{i(i+1)}m_{(i+1)j}=m_{ij}b_{jj} - a_{ii}m_{ij} + m_{i(j-1)}b_{(j-1)j}, 
\end{align}
where 
\begin{align*}
m_{i0}= m_{(n+1)j}=0 \> \> \> \text{for} \> \> i,j=1, \dots, n.
\end{align*}
For $i=1$, we have from Eq.(\ref{firstrow}) that $(m_{1j})_{j=2}^n=(0,\dots, 0)$. By substituting in Eq.(\ref{diffequ1}) we obtain $(m_{2j})_{j=3}^n=(0,\dots, 0)$.\\
For $i=2$, since $(m_{2j})_{j=3}^n=(0,\dots, 0)$, by substituting in Eq.(\ref{diffequ1}), we obtain $(m_{3j})_{j=4}^n=(0,\dots, 0)$.\\
By using the principal of mathematical induction we can easily show that $(m_{ij})_{j=i+1}^n=(0,\dots, 0)$ for $i=3,\dots,n$. Therefore, all the elements above the main diagonal of the matrix $M$ are zero which makes it a lower triangular matrix. 
\end{proof}

\begin{example}
 Using the results in Example \ref{exp1}, we can find a $3\times 3$ matrix $M$ to show that the two estimated CPH distributions in Table 1  are equivalent.  In order to find $M$, we solve conditions (a), (b) and (c) in Theorem \ref{theorem1}. The number of unknowns is  $(n^2+n-2)/2$ (where $n=3$) since $M$ is lower triangular and the first row of $M$ is known. This generates a system of $(n^2+n-2)/2$ equations and we obtain
\begin{align*}
M=\begin{bmatrix}
        1 & 0 & 0  \\
      0.7 &  0.3 & 0 \\
       0 & 0 & 1
       \end{bmatrix}.
\end{align*}
It is clear that $M$ satisfies Eq.(\ref{m1=1}) where the rows sum to one.
 \end{example}
 
\begin{corollary} 
\label{corollary2}
  If $(p_a, Q_a)$ and $(p_b, Q_b)$ are two equivalent CPH distributions then the absorbing rates from the first phase of their corresponding Markov chains, are identical. In other words, the first elements of the column vectors $q_a$ and $q_b$ are equal. 
  \end{corollary}
  
\begin{proof} Denote by  $\mu_{a1}$ and $\mu_{b1}$ the first elements of the absorbing rate vectors  $q_a$ and $q_b$ respectively. As $(p_a,Q_a)$ and  $(p_b,Q_b)$ are equivalent then  condition (c) of Theorem \ref{theorem1} is satisfied. We have
\begin{align*}
Mq_b=q_a
 \Rightarrow (1,0,\dots, 0)\times q_b=\mu_{a1}
 \Rightarrow \mu_{b1}=\mu_{a1}.
\end{align*} 
As we have already mentioned, this property is  seen in Table 1 where $\hat{\mu_1}$ remains invariant in the two estimated equivalent distributions. 
\end{proof}


 \begin{theorem}
 \label{theorem2}
{ \it If two $n$-CPH distributions $( p_a,Q_a)$  and $( p_b,Q_b)$ are equivalent with $Q_a \neq Q_b$, then $Q_a$ and $Q_b$ have the same but permuted diagonal entries. }
\end{theorem}

\begin{proof}
As seen in the proof of Theorem \ref{theorem1}, if two PH distributions are equivalent then the generators $Q_a$ and $Q_b$  have equal Jordan matrices (Eq.(\ref{jordan})), which means they have same eigenvalues. 

 Since $Q_a$ and $Q_b$ are triangular matrices then their eigenvalues are on their diagonals. Therefore, $Q_a$ and $Q_b$ have same diagonal entries. We still need to prove that the diagonal entries  of two equivalent Coxian distributions cannot have the same position and that they are permuted. We prove this by contradiction.

Assume that the position of the diagonal elements are the same, i.e. $a_{ii}=b_{ii}$ for $i=1, \dots, n$. Since $(p_a,Q_a)$ is equivalent to $(p_b,Q_b)$ then  $MQ_b=Q_aM$. We consider again Equation (\ref{diffequ1}) which reduces to
\begin{align}\label{diffequ2}
a_{i(i+1)}m_{(i+1)j}&= m_{i(j-1)}b_{(j-1)j} 
\end{align}
where 
\begin{align*}
m_{i0}= m_{(n+1)j}=0 \> \> \> \text{for} \> \> i,j=1, \dots, n.
\end{align*}
For $j=1$, since $(m_{i0})_{i=1}^n=(0,\dots,0)$, it follows from  Eq.(\ref{diffequ2}) that $(m_{(i+1)1})_{i=1}^{n}=(0,\dots,0)$ or equivalently  $(m_{i1})_{i=2}^{n}=(0,\dots,0)$.\\
For $j=2$, since $(m_{i1})_{i=2}^{n}=(0,\dots,0)$, then from  Eq.(\ref{diffequ2})  $(m_{i2})_{i=3}^{n}=(0,\dots,0)$.\\
We can show by induction that the rest of the elements under the main diagonal are all zero.  This makes the matrix $M$ upper triangular, but, from Corollary \ref{corollary1}, we know that the matrix is lower triangular. Therefore, $M$ is a diagonal matrix. Since the rows of $M$ have to sum to one, then $M$ is an identity matrix. It then follows from the relation, $MQ_b=Q_aM$, that $Q_a=Q_b$. This contradicts with the given which is $Q_a \neq Q_b$.

Therefore, the assumption that the position of the diagonal elements is the same is false. Thus $Q_a$ and $Q_b$ have permuted diagonal entries.  
\end{proof}

\begin{example} 
\label{exp3}
We use again the two estimated equivalent distributions in Example \ref{exp1}. The generator matrices are
\begin{align*}
Q_a=\begin{bmatrix}
        -0.571 & 0.570 & 0  \\
      0 &  -0.172 & 0.029 \\
       0 & 0 & -0.091
       \end{bmatrix} \text{and}    \> \>     
       Q_b=\begin{bmatrix}
        -0.172 & 0.170 & 0  \\
      0 &  -0.571 & 0.096 \\
       0 & 0 & -0.091
       \end{bmatrix}.
\end{align*}
We can see that the two sets of eigenvalues  are the same but in a different position with different superdiagonal entries. 
\end{example}
 
 \section{Construction of all equivalent representations}
Based on Corollary ~\ref{corollary2} and Theorem ~\ref{theorem2},  we present an analytical algorithm that produces and checks the feasibility of all the possible representations of a pre-defined CPH distribution. This algorithm is completely mathematical and totally independent of the fitting process. The approach only requires one representation to generate all others, i.e., we need only fit the model once and we are guaranteed to produce all representations of the CPH model.

For a given $n$-CPH distribution  $(p_a,Q_a)$ with generator $Q_a$ as defined in (\ref{QaQb}), we have the vector of diagonal elements  diag$(Q_a)=(a_{11},a_{22} \dots, a_{nn})$. Let the  vector  $\mathcal{V}^{(r)}$, $r=1, 2, \dots, n!$, be a permutation of the components of diag$(Q_a)$, where $\mathcal{V}^{(1)}$ is the identity permutation, i.e.,  $\mathcal{V}^{(1)}=$ diag$(Q_a)$. We denote $\mathcal{V}_i^{(r)}$ the $i$th component of the vector $\mathcal{V}^{(r)}$.\\
 
\begin{algo} 
For an arbitrary permutation $\mathcal{V}^{(r)}$, $r=2, \dots, n!$, suppose there exists $(p_b,Q_b)$ that is  equivalent to  $( p_a,Q_a)$. We know that $ p_a= p_b=(1,0,\dots, 0)$ and from Theorem \ref{theorem2} we have that the diagonal entries of $Q_b$ are the same as the diagonal entries of $Q_a$ but permuted, which means  diag$(Q_b$)=$\mathcal{V}^{(r)}$.  The generator $Q_b$ could be written as
 
\begin{align*}
Q_b = \begin{bmatrix}
\mathcal{V}_1^{(r)} & b_{12} & & 0\\
 & \ddots & \ddots & \\
&  & \ddots & b_{(n-1)n} \\
0 & &  & \mathcal{V}_n^{(r)} 
\end{bmatrix},
\end{align*}

where the vector $\{b_{12}, \dots, b_{(n-1)n}\}=$ superdiag$(Q_b)$.\\
 To find the equivalent distribution $(p_b,Q_b)$, it is sufficient to find the elements of the vector $\{b_{12}, \dots, b_{(n-1)n}\}$.
  
\begin{step}  The first component, $b_{12}$, can be found using Corollary \ref{corollary2} and Equation (\ref{vecq}). We have
  \begin{align*}
 \mathcal{V}_1^{(r)} +b_{12}&= a_{11} +a_{12} \\
 \Rightarrow  b_{12}&=-\mathcal{V}_1^{(r)}+a_{11} + a_{12} \\
  \end{align*}
  
  \end{step}
\begin{step}  To find the remaining $n-2$ unknowns of the vector  $\{b_{12}, \dots, b_{(n-1)n}\}$, we use the fact that the two equivalent distributions have identical moments and we solve the following system of  $n-2$ non-linear equations
 \begin{align}
 p_aQ_a^{-1}\mathbbm{1}&=p_bQ_b^{-1}\mathbbm{1} \nonumber\\
 p_aQ_a^{-2}\mathbbm{1}&=p_bQ_b^{-2}\mathbbm{1}  \label{simulequat}\\
 &  \vdots \nonumber\\
 p_aQ_a^{-(n-2)}\mathbbm{1}&=p_bQ_b^{-(n-2)}\mathbbm{1}.\nonumber
 \end{align}
 \end{step}
 
\begin{step} Repeat the above steps with a different permutation $\mathcal{V}^{(r)}$, $r=2, \dots, n!$.\\
\end{step}
\end{algo}

The system has a unique solution. This follows from the proof of Theorem \ref{theorem2} where we showed that if two generator matrices have same diagonals then they are equal. Thus, it is impossible to find multiple matrices $Q_b$ that are similar to  $Q_a$ and have identical diagonals and different superdiagonals. \\
For each permutation we accept the solution that falls within the constraints $0< b_{i(i+1)} \leq -\mathcal{V}_i^{(r)}$, for $i=1, \dots, n-1$. A solution outside the constraints results in a non-Markovian  representation  \citep{horvath2016} with negative transition rates or negative  absorbing rates, which does not lie in the Coxian space that we are  searching. If no permutation obeys these constraints, then the CPH distribution $(p_a, Q_a)$ has a unique representation. We illustrate the algorithm with the following examples.

\begin{example}  Given the 3-CPH distribution with generator 
\begin{align*} 
 Q_a=\begin{bmatrix}
        -1.0018 & 1 & 0  \\
      0 &  -0.2138 & 0.211 \\
       0 & 0 & -0.0259
       \end{bmatrix}. 
\end{align*}
The algorithm found that all the six permutations of the diagonal elements are feasible. The solutions are shown in Table ~\ref{table2},  where the row $r=1$ corresponds to the identity permutation.

\begin{table}[t]
 \centering
 \caption{All the possible representations of a CPH distribution.}
 \begin{tabular}{ccccccccc} 
 \hline
 &&&&&&&&\\[-0.3cm]
$r$ & $\mathcal{V}_1^{(r)}$ &$\mathcal{V}_2^{(r)}$ &$\mathcal{V}_3^{(r)}$ &$b_{12}$ & $b_{23}$ & $LoS_1$ & $LoS_2$ & $LoS_3$ \\ 
\hline
 &&&&&&&&\\[-0.4cm]
1& -1.0018 & - 0.2138&   -0.0259 &  1	&0.211 & 1& 4.7& 38.6\\
2 &  -0.0259&   -1.0018 &  -0.2138	&0.0241&0.941&38.6&	1&4.7\\
3 &  -0.2138 & - 0.0259 &-1.0018	& 0.212 &0.0175 & 4.7&38.6&1\\
4&   -0.2138 &  -1.0018 & -0.0259	&0.212 & 0.993& 4.7& 1&38.6 \\
5&   -1.0018 &  -0.0259 & -0.2138	&1 &0.0227 &1&	38.6&1\\
6& -0.0259  & -0.2138 &-1.0018	&0.0241 &0.154 &38.6 &	4.7&1\\
&&&&&&&&\\[-0.4cm]
\hline
\end{tabular}
\label{table2}
\end{table}
\end{example}

\begin{example} For the matrices $Q_a$ and $Q_b$ defined in Example \ref{exp3}, we found that they are equivalent to each other with no other existing equivalent distributions.
\end{example}

 \begin{example}
  Given the 4-CPH distribution with generator
   \begin{align*}
 Q_a=\begin{bmatrix}
        -1 & 0.965 & 0  &0 \\
      0 &  -0.447 & 0.435 & 0 \\
       0 & 0 &-0.446 &0.120 \\ 
       0 & 0 &0 &-0.151
       \end{bmatrix}
\end{align*}
This distribution was found to have no equivalent distributions. It is a Coxian distribution with a unique representation.
\end{example}

\section{Discussion}
In this paper, we have provided deeper insight into the non-uniqueness in the representation of CPH distributions. Our new results have facilitated the development of a mechanism  that finds all the possible representations and verifies their feasibility. 

For large values of $n$, it is clear that the algorithm becomes computationally expensive as it will have to run $n!$ times (albeit these can be run in parallel). However, solving the system of equations involved is far less intensive than attempting to refit the model multiple times from different initial values (with many attempts failing to converge due to the complicated likelihood surface), and this is especially true when the sample size is large.  Moreover, unlike the ad-hoc procedure of fitting from different initial values, our proposed algorithm is guaranteed to find all representations and only requires one vector of parameters as its input, i.e., the estimation routine need only converge once, after which our algorithm operates independently of the data and estimation routine. Lastly, it is worth noting that applications of CPH models in the literature rarely require a very large number of states to adequately fit data (typically $n < 10$ is sufficient) so that our proposal is feasible in most practical purposes.

\newpage

\bibliographystyle{apa}
\bibliography{References_paper1}

\end{document}